\documentclass[a4paper]{amsart}
\usepackage{amssymb}
\usepackage{amsmath}
\usepackage{ifpdf}
\ifpdf
\usepackage[pdftex]{graphicx}
\DeclareGraphicsExtensions{.pdf}
\usepackage[pdftex,plainpages=false,pdfpagelabels,linktocpage,pdfstartview=FitH,colorlinks=true,linkcolor=blue,citecolor=magenta]{hyperref}
\else
\usepackage[dvips]{graphicx}
\usepackage[hypertex,linkcolor=cyan]{hyperref}
\fi
\usepackage[initials,msc-links]{amsrefs}[2007/10/22]
\newcommand{\R}{\mathbb{R}}
\newcommand{\Z}{\mathbb{Z}}
\newcommand{\N}{\mathbb{N}}

\newcommand{\half}{\tfrac{1}{2}}
\newcommand{\im}{\mathrm{Im}}
\newcommand{\del}{\partial}
\newcommand{\bk}{\mathbf{k}}
\DeclareMathOperator{\der}{d}

\newcommand{\D}{\mathrm{d}}

\newtheorem{theorem}{Theorem}[section]
\newtheorem{corollary}[theorem]{Corollary}
\newtheorem{lemma}[theorem]{Lemma}
\newtheorem{proposition}[theorem]{Proposition}
\theoremstyle{definition}

\newtheorem{definition}[theorem]{Definition}
\theoremstyle{plain}
\numberwithin{equation}{section}

\begin{document}
\title{Formal conserved quantities for isothermic surfaces}
\author{F.E. Burstall}
\address{Department of Mathematical Sciences\\ University of Bath\\
Bath BA2 7AY\\UK} \email{feb@maths.bath.ac.uk}
\author{S.D. Santos}
\address{Universidade de Lisboa\\ Faculdade de Ci\^encias\\ Departamento
de Matem\'atica\\ CMAF\\1749-016 Lisboa\\Portugal}
\email{susantos@ptmat.fc.ul.pt}
\begin{abstract}
Isothermic surfaces in $S^n$ are characterised by the existence of a
pencil $\nabla^t$ of flat connections.  Such a surface is special of
type $d$ if there is a family $p(t)$ of $\nabla^t$-parallel sections
whose dependence on the spectral parameter $t$ is polynomial of
degree $d$.  We prove that any isothermic surface admits a family of
$\nabla^t$-parallel sections which is a formal Laurent series in
$t$.  As an application, we give conformally invariant
conditions for an isothermic surface in $S^3$ to be special.
\end{abstract}

\maketitle

\textbf{Keywords:} special isothermic surfaces, polynomial and formal conserved quantities.

\medskip

\textbf{MSC:} 53A30, 53A05.

\section*{Introduction}

Isothermic surfaces, that is, surfaces which admit conformal
curvature line coordinates, were intensively studied around the turn
of the 20th century by Darboux, Bianchi and others
\cite{Dar99b,Bia05,Cal03,Cal15}.  These classical works revealed a
rich transformation theory that has been revisited in modern times
from the viewpoint of integrable systems
\cite{Bur06,CieGolSym95,Sch01}.  At the heart of the integrable
systems formalism is the observation that there is a pencil of flat
connections $\nabla^{t}=\D+t\eta$, $t\in\R$, associated to each
isothermic surface.

In our previous work \cite{BurSan12}, we distinguished the class of
\emph{special isothermic surfaces} (of type $d$) which are
characterised by the existence of a \emph{polynomial conserved
quantity}, that is, a family $p(t)$ of $\nabla^{t}$-parallel sections
whose dependence on $t$ is polynomial (of degree $d$).  The existence
of such a polynomial conserved quantity amounts to a differential
equation on the principal curvatures of the surface.  For example, an
isothermic surface in $S^3$ is special of type $1$ if it has
constant mean curvature with respect to a constant curvature metric
on (an open subset of) $S^3$ and special of type $2$ if it is a
special isothermic surface in the sense of Darboux and Bianchi
\cite{Bia05,Dar99}---a class of isothermic surfaces that originally
arose in the study of surfaces isometric to a quadric.

The purpose of the present paper is to answer a question posed to one
of us by Nigel Hitchin: does any isothermic surface in $S^{n}$ admit a
\emph{formal conserved quantity}, thus a solution $p(t)$ of $\nabla^tp(t)=0$
with $p(t)=\sum_{i\leq 0} p_{i}t^{i}$ a formal Laurent series?  We
give an affirmative answer locally, away from the (discrete) zero-set
of $\eta$, and globally when $n=3$.  In particular, any isothermic
$2$-torus in $S^3$ admits a formal conserved quantity.  This is an
analogue of the existence of formal Killing fields for harmonic maps
\cite{BurFerPedPin93} although the method is somewhat different since
here we deal with nilpotent rather than semisimple gauge potentials.

Since polynomial conserved quantities are also formal conserved
quantities, our arguments allow us to give conformally invariant
conditions, in terms of the Schwarzian derivative and Hopf
differential introduced in \cite{BurPedPin02}, for an isothermic
surface to be special of type $d$.  We illustrate these results with
the case of surfaces of revolution and other equivariant surfaces
where these conditions amount to a differential equation on the
curvature of a profile curve.

Some of the following results can also be found in the second author's doctoral thesis \cite{San08}, using a different approach.

\section{Preliminaries}

\subsection{The conformal sphere}
\label{sec:conformal-sphere}

We will study isothermic surfaces in the $n$-sphere from a
conformally invariant view-point and so use Darboux's light-cone
model of the conformal $n$-sphere.  For this, contemplate the light-cone
$\mathcal{L}$ in the Lorentzian vector space $\mathbb{R}^{n+1,1}$ and
its projectivisation $\mathbb{P}(\mathcal{L})$.  This last has a
conformal structure where representative metrics $g_{\sigma}$ arise from
never-zero sections $\sigma$ of the tautological bundle
$\pi:\mathcal{L}\to\mathbb{P}(\mathcal{L})$ via
\begin{equation*}
g_{\sigma}(X,Y)=(\D\sigma(X),\D\sigma(Y)).
\end{equation*}
Then $S^n\cong \mathbb{P}(\mathcal{L})$ \emph{qua} conformal
manifolds.  Indeed, for non-zero $w\in\mathbb{R}^{n+1,1}$, let $E(w)$
be the conic section given by
\begin{equation*}
E(w)=\{v\in\mathcal{L}:(v,w)=-1\}
\end{equation*}
with (definite) metric induced by the ambient inner product on
$\mathbb{R}^{n+1,1}$.  Then $\pi_{|E(w)}$ is a conformal
diffeomorphism onto its image.  In particular, when $w_{0}$ is unit
time-like, we have an isometry $x\mapsto x+w_{0}$ from the unit sphere in
$\langle w_{0}\rangle^{\perp}$ to $E(w_0)$ and thus a conformal
diffeomorphism from that sphere to $\mathbb{P}(\mathcal{L})$.  More
generally, $E(w)$ has constant sectional curvature $-(w,w)$.

\subsection{Invariants of a conformal immersion}
\label{sec:invar-conf-immers}
Let $\Sigma$ be a Riemann surface and $\Lambda:\Sigma\to
S^{n}\cong\mathbb{P}(\mathcal{L})$ a conformal immersion. We view
$\Lambda$ as a null line subbundle of the trivial bundle
$\underline{\mathbb{R}}^{n+1,1}=\Sigma\times \mathbb{R}^{n+1,1}$.

The \emph{central sphere congruence} assigns, to each $x\in \Sigma$,
the unique $2$-sphere $S(x)$ tangent to $\Lambda$ at the point
$\Lambda(x)$, which has the same mean curvature vector as $\Lambda$
at $x$.  Having in mind the identification between $2$-dimensional subspheres
of $S^{n}$ and $(3,1)$-planes of $\mathbb{R}^{n+1,1}$ via
\begin{equation*}
V\mapsto\mathbb{P}(\mathcal{L}\cap V),
\end{equation*}
the central sphere congruence of $\Lambda$ amounts to a subbundle $V$
of $\underline{\mathbb{R}}^{n+1,1}$, with signature $(3,1)$.

Fix a holomorphic coordinate $z=u+iv$ on $\Sigma$ and take the unique (up to sign)
lift $\psi\in\Gamma\Lambda$ such that $$|\D \psi|^{2}=|\D z|^{2}.$$  Then
\[V\otimes\mathbb{C}=\langle
\psi,\psi_{z},\psi_{\bar{z}},\psi_{z\bar{z}}\rangle.\]
Consider now the unique section $\hat{\psi}\in\Gamma(V)$ such that $$(\hat{\psi},\hat{\psi})=0,\;(\hat{\psi},\psi)=-1 \mbox{ and }\;(\hat{\psi},\D\psi)=0,$$
which provides a new frame for $V\otimes\mathbb{C}$, namely
$$\psi,\psi_{z},\psi_{\bar{z}}\mbox{ and }\hat{\psi}.$$

According to \cite{BurPedPin02}, we have
$$\psi_{zz}+\frac{c}{2}\psi=\kappa,$$
for a complex function $c$ and
$\kappa\in\Gamma(V^{\perp}\otimes\mathbb{C})$.  These latter
invariants are, respectively, the \emph{Schwarzian derivative} and
\emph{Hopf differential} of $\Lambda$ with respect to $z$ and,
together with the connection $D$ on $V^{\perp}$ given by
orthoprojection of flat differentiation, determine $\Lambda$ up to
conformal diffeomorphisms of $S^n$.

Our frame satisfies:
\begin{equation}
\label{eq:6}
\begin{split}
\psi_{zz}&=-\frac{c}{2}\psi+\kappa\\
\psi_{z\bar{z}}&=-(\kappa,\bar{\kappa})\psi+\frac{1}{2}\hat{\psi}\\
\hat{\psi}_{z}&=-2(\kappa,\bar{\kappa})\psi_{z}-c\psi_{\bar{z}}+2D_{\bar{z}}\kappa\\
\xi_{z}&=2(\xi,D_{\bar{z}}\kappa)\psi-2(\xi,\kappa)\psi_{\bar{z}}+D_{z}\xi,
\end{split}
\end{equation}
for each $\xi\in\Gamma(V^{\perp}\otimes\mathbb{C})$.  The
corresponding structure equations are the
\emph{conformal Gauss equation}:
\begin{equation*}
\frac{1}{2}c_{\bar{z}}=3(\kappa,D_{z}\bar{\kappa})+(\bar{\kappa},D_{z}\kappa);
\end{equation*}
the \emph{conformal Codazzi equation}:
$$\mathrm{Im}(D_{\bar{z}}D_{\bar{z}}\kappa+\frac{1}{2}\bar{c}\kappa)=0$$
and the \emph{conformal Ricci equation}:
$$D_{\bar{z}}D_{z}\xi-D_{z}D_{\bar{z}}\xi
-2\langle \xi,\kappa\rangle \bar{\kappa}+2\langle \xi,\bar{\kappa}\rangle \kappa=0.$$

See \cite{BurPedPin02} for more details.

\subsection{Isothermic and special isothermic surfaces}
\label{sec:isoth-spec-isoth}
Classically, an isothermic surface is a surface in $S^n$ that admits
conformal curvature line coordinates but we shall follow
\cite{BurDonPedPin11,Her03} and adopt the following conformally
invariant formulation:
\begin{definition}
An immersion $\Lambda:\Sigma\longrightarrow
S^{n}\cong\mathbb{P}(\mathcal{L})$, $\Lambda$ is an \emph{isothermic
surface} if there is a non-zero closed $1$-form
$\eta\in\Omega^{1}\otimes o(\mathbb{R}^{n+1,1})$ taking values in
$\Lambda\wedge\Lambda^{\perp}$ \footnote{Recall the isomorphism
$\bigwedge^{2}\mathbb{R}^{n+1,1}\cong o(\mathbb{R}^{n+1,1})$ via
$(u\wedge v)w=(u,w)v-(v,w)u$, for all
$u,v,w\in\mathbb{R}^{n+1,1}$.}.
\end{definition}
One makes contact with the classical formulation by defining
$q\in\Gamma(S^2T^{*}\Sigma)$ by
\begin{equation*}
q(X,Y)\sigma=\eta_X\D_Y\sigma,
\end{equation*}
for any $\sigma\in\Gamma\Lambda$.  Then $d\eta=0$ if and only if $q$
is a holomorphic quadratic differential which commutes with the
second fundamental form of $\Lambda$.  Now $q$ and hence $\eta$
vanishes only on a discrete set and, off that set, we can find a
holomorphic coordinate $z$ such that $q=\D z^2$.  In terms of the
corresponding lift $\psi$ of Section~\ref{sec:invar-conf-immers}, we
then have
\begin{equation*}
\eta=-\psi\wedge(\psi_{\bar{z}}\D z+\psi_{z}\D \bar{z})
\end{equation*}
which commutes with the second fundamental form if and only if
$\kappa$ is real.  In this case, $z=u+iv$ where $u,v$ are curvature
line coordinates.

The conformal Gauss and Codazzi equations are now given simply by
\begin{equation*}
c_{\bar{z}}=4(\kappa,\kappa)_{z}
\mbox{ and }
\mathrm{Im}(D_{\bar{z}}D_{\bar{z}}\kappa+\frac{1}{2}\bar{c}\kappa)=0,
\end{equation*}
while the conformal Ricci equation amounts to the familiar assertion
that the connection $D$ on $V^{\perp}$ is flat.

The key to the integrable systems theory of isothermic surfaces is
the observation that the family of metric connections $\D+t\eta$, $t\in\R$ on
$\underline{\R}^{n+1,1}$ are \emph{flat} and so have a good supply of
parallel sections.  In \cite{BurSan12}, we considered isothermic
surfaces which admitted parallel sections with polynomial dependence
on $t$ and so introduced the notion of special isothermic
surfaces of type $d\in\mathbb{N}_{0}$:
\begin{definition}
An isothermic surface $(\Lambda,\eta)$ in $S^{n}$ is a \emph{special
isothermic surface of type $d\in\mathbb{N}_{0}$} if there is a
polynomial $p(t)=\sum_{i=0}^d p_{i}t^{i}\in\Gamma(\underline{\mathbb{R}}^{n+1,1})[t]$ of
degree $d$ such that $(\D+t\eta)p(t)\equiv 0$.  We call such a $p(t)$
a \emph{polynomial conserved quantity} of $(\Lambda,\eta)$.
\end{definition}

An immediate consequence of $(\D+t\eta)p(t)\equiv 0$ is that $\D
p_0=0$ so that $p_0$ is constant and therefore, if non-zero, defines
a conic section as in Section~\ref{sec:conformal-sphere}.  We
therefore refine our definition:
\begin{definition}
An isothermic surface $(\Lambda,\eta)$ in $S^{n}$ is a
\emph{special isothermic surface of type $d\in\mathbb{N}_{0}$ in
$E(w)$} if $(\Lambda,\eta)$ admits a polynomial conserved quantity
$p(t)=\sum_{i=0}^d p_{i}t^{i}$ of degree $d$ with $p_{0}\in\langle
w\rangle$.
\end{definition}

The condition that an isothermic surface $(\Lambda,\eta)$ be special
of type $d$ amounts to a differential equation on the principal
curvatures of $\Lambda$.  For example, generically, $(\Lambda,\eta)$ is a special
isothermic surface of type $1$ in $E(w)$ if and only if the lift
$F:\Sigma\longrightarrow E(w)$ of $\Lambda$ is a generalised
$H$-surface (which, in codimension $1$, amounts to the mean curvature
$H$ being constant) (\cite{BurCal}, see also \cite{BurSan12}).
Again, in \cite{BurSan12}, we show that $(\Lambda,\eta)$ is special
of type $2$ in $E(w)$ if and only if there are real constants $A$,
$B$ and $C$ such that
\begin{equation}\label{new Bianchi condition}
\begin{cases}
H_{uu}+\theta_{u}H_{u}-\theta_{v}H_{v}-\frac{1}{2}Mk_{1}-Ak_{1}-Be^{-2\theta}+C-\frac{1}{2}L(w,w)=0\\
H_{vv}-\theta_{u}H_{u}+\theta_{v}H_{v}+\frac{1}{2}Mk_{2}+Ak_{2}-Be^{-2\theta}-C+\frac{1}{2}L(w,w)=0,
\end{cases}
\end{equation}
where $z=u+iv$ is a holomorphic coordinate for which
$\eta=-\psi\wedge(\psi_{\bar{z}}\D z+\psi_{z}\D \bar{z})$,
\begin{equation*}
I=e^{2\theta}(\D u^{2}+\D v^{2})\mbox{ and }
I\!I=e^{2\theta}(k_{1}\D u^{2}+k_{2}\D v^{2})
\end{equation*}
are, respectively, the first and second fundamental forms of the lift
$F:\Sigma\longrightarrow E(w)$ of $\Lambda$,
$H=\frac{k_{1}+k_{2}}{2}$ is the mean curvature of $F$,
$L=e^{2\theta}(k_{1}-k_{2})$ and $M=-HL$.  This condition amounts to
the surface being a special isothermic surface in the sense of
Darboux and Bianchi \cite{Bia05, Dar99b}, at least when $H_uH_v$ is
non-zero (see \cite{BurSan12}).

\section{Formal conserved quantities}\label{sec:fcq}

Denote by $\Gamma(\underline{\mathbb{R}}^{n+1,1})[[t,t^{-1}]]$ the
vector space of the formal Laurent series in $t$ with coefficients
in $\Gamma(\underline{\mathbb{R}}^{n+1,1})$, i.e., the series
$\sum_{k\leq s} p_{k}t^{k}$, for some $s\in\mathbb{N}_{0}$, with
coefficients in $\Gamma(\underline{\mathbb{R}}^{n+1,1})$.

We define a $\R[[t,t^{-1}]]$ valued inner product on
$\Gamma(\underline{\mathbb{R}}^{n+1,1})[[t,t^{-1}]]$ by
\begin{equation*}
\bigl(p(t),q(t)\bigr):=\sum_{k\leq s+r}\sum_{\substack{i\leq s,j\leq
r\\i+j=k}}(p_{i},q_{j})t^{k},
\end{equation*}
for all $p(t)=\sum_{k\leq s} p_{k}t^{k},q(t)=\sum_{k\leq r}
q_{k}t^{k}\in\Gamma(\underline{\mathbb{R}}^{n+1,1})[[t,t^{-1}]]$.

\begin{definition}Let $(\Lambda,\eta)$  be an isothermic surface in $S^{n}$ and let
$p(t)=\sum_{i\leq 0}
p_{i}t^{i}\in\Gamma(\underline{\mathbb{R}}^{n+1,1})[[t,t^{-1}]]$
such that $p_{0}\neq0$. We say that $p(t)$ is a \emph{formal
conserved quantity} of $(\Lambda,\eta)$ if $(\D+t\eta)p(t)$ is the
zero series.
\end{definition}

\begin{proposition}\label{top term}
Let $(\Lambda,\eta)$ be an isothermic surface in $S^{n}$. If
$p(t)=\sum_{i\leq 0}
p_{i}t^{i}$ is a formal conserved quantity of $(\Lambda,\eta)$, then
\begin{enumerate}
\item\label{item} $p_{0}$ is a $D$-parallel section of $V^{\perp}$;
\smallskip
\item the series $\bigl(p(t),p(t)\bigr)\in\mathbb{R}[[t,t^{-1}]]$ (thus
independent of $x\in\Sigma$).
\end{enumerate}
\end{proposition}
\begin{proof}
Item (\ref{item}) is proved in \cite[Proposition~2.2]{BurSan12}.

That the coefficients of $\bigl(p(t),p(t)\bigr)$ are constant follows from the
fact that $\D+t\eta$ is a metric connection:
\begin{equation*}
\begin{split}
\der\bigl(p(t),p(t)\bigr)&=\sum_{k\leq 0}\sum_{\substack{i,j\leq
0\\i+j=k}}\der(p_{i},p_{j})t^{k}\\
&=\sum_{k\leq 0}\sum_{\substack{i,j\leq
0\\i+j=k}}\big(((\der+t\eta)p_{i},p_{j})+(p_{i},(\der+t\eta)p_{j})\big)t^{k}\\
&=((\D+t\eta)p(t),p(t)\bigr)+\bigl(p(t),(\D+t\eta)p(t)\bigr)=0.
\end{split}
\end{equation*}
\end{proof}

The condition that $p(t)$ be a formal conserved quantity amounts to a
recursive system of equations on its coefficients which we now
describe in terms of the frame
$\psi,\psi_{z},\psi_{\bar{z}},\hat{\psi}$ of $V$ associated with a
holomorphic coordinate $z$ as in Section~\ref{sec:isoth-spec-isoth}.
\begin{proposition}\label{necessary}
Let $(\Lambda,\eta)$ be an isothermic surface in $S^{n}$ and let
$z=u+iv$ be a holomorphic coordinate on $\Sigma$ such that
$\eta=-\psi\wedge(\psi_{\bar{z}}\D z+\psi_{z}\D \bar{z})$.

Let $p(t)=\sum_{i\leq 0}
p_{i}t^{i}\in\Gamma(\underline{\mathbb{R}}^{n+1,1})[[t,t^{-1}]]$ and,
for $i\in\mathbb{Z}^{-}_0$, write
\begin{equation*}
p_{i}=\alpha_{i} \psi+\beta_{i}\psi_{z}+\bar{\beta_{i}}\psi_{\bar{z}}+\gamma_{i}\hat{\psi}+q_{i},
\end{equation*}
where each $\alpha_i,\gamma_i$ is a real function, $\beta_i$ is a
complex function and $q_{i}\in\Gamma(V^{\perp})$.  Then $p(t)$ is a
formal conserved quantity if and only if, for all $i\in\mathbb{Z}^{-}_0$,
\begin{subequations}\label{eq:1}
\begin{align}
    \label{eq:2}\beta_{i}&=-2\gamma_{i,\bar{z}}\\
    \label{eq:3}\alpha_{i}&=2\gamma_{i,z\bar{z}}+2(\kappa,\kappa)\gamma_{i}\\
    \label{eq:4}D_{z}q_{i}&=2\gamma_{i,\bar{z}}\kappa-2\gamma_{i}D_{\bar{z}}\kappa\\
    \label{eq:5}\gamma_{i-1}&=2\gamma_{i,zz}+c\gamma_{i}+2(q_{i},\kappa).
\end{align}
\end{subequations}
\end{proposition}
\begin{proof}
For each $i\in\Z^{-}_{0}$, we have, using \eqref{eq:6},
\begin{equation*}
\begin{split}
p_{i,z}+\eta_{\frac{\partial}{\partial z}}p_{i-1}&=\big(\alpha_{i,z}-\frac{c\beta_{i}}{2}-\bar{\beta_{i}}(\kappa,\kappa)+2(q_{i},\kappa_{\bar{z}})+\frac{\beta_{i-1}}{2}\big)\psi\\
&+\big(\alpha_{i}+\beta_{i,z}-2(\kappa,\kappa)\gamma_{i}\big)\psi_{z}
+\big(\bar{\beta_{i}}_{,z}-c\gamma_{i}-2(q_{i},\kappa)+\gamma_{i-1}\big)\psi_{\bar{z}}\\
&+\big(\frac{\bar{\beta_{i}}}{2}+\gamma_{i,z}\big)\hat{\psi}\\
&+\big(\beta_{i}\kappa+2\gamma_{i} D_{\bar{z}}\kappa+D_{z}q_{i}\big).
\end{split}
\end{equation*}
The vanishing of the $\hat{\psi}$ coefficient is equivalent to
$\beta_{i}=-2\gamma_{i,\bar{z}}$ since $\gamma_{i}$ is real and then
\eqref{eq:3} and \eqref{eq:5} amount to the vanishing
of the coefficients of $\psi_z$, $\psi_{\bar{z}}$ respectively while
\eqref{eq:4} is the same as the vanishing of the normal component.
We are left with the $\psi$ component but its vanishing is a
differential consequence of \eqref{eq:1}.  Indeed:
\begin{equation*}
\begin{split}
\beta_{i-1}&=-2\gamma_{i-1,\bar{z}}=-2(2\gamma_{i,zz\bar{z}}+c_{\bar{z}}\gamma_{i}+c\gamma_{i,\bar{z}}+2(q_{i,\bar{z}},\kappa)+2(q_{i},\kappa_{\bar{z}}))\\
&=-2(\alpha_{i,z}-2(\kappa,\kappa)_{z}\gamma_{i}-2(\kappa,\kappa)\gamma_{i,z}+c_{\bar{z}}\gamma_{i}+c\gamma_{i,\bar{z}}+2(q_{i,\bar{z}},\kappa)+2(q_{i},\kappa_{\bar{z}}))\\
&=-2(\alpha_{i,z}+2(\kappa,\kappa)_{z}\gamma_{i}+(\kappa,\kappa)\bar{\beta_{i}}-\frac{c\beta_{i}}{2}+2(q_{i,\bar{z}},\kappa)+2(q_{i},\kappa_{\bar{z}}))\\
\end{split}
\end{equation*}
while
\begin{equation*}
(q_{i,\bar{z}},\kappa)=(-\bar{\beta_{i}}\kappa-2\gamma_{i}\kappa_{z},\kappa)=-\bar{\beta_{i}}(\kappa,\kappa)-\gamma_{i}(\kappa,\kappa)_{z}.
\end{equation*}
\end{proof}

The equations \eqref{eq:1} amount to a recursive scheme for
constructing a formal conserved quantity starting with $\gamma_0$ so
long as we can be assured that each $\gamma_{i-1}$ defined by
\eqref{eq:5} is real and that \eqref{eq:4} is solvable for each $i$.
For this, we need:
\begin{lemma}\label{lemma1}
Let $(\Lambda,\eta)$ be an isothermic surface in $S^{n}$ and let
$z=u+iv$ be a holomorphic coordinate on $\Sigma$ such that
$\eta=-\psi\wedge(\psi_{\bar{z}}\D z+\psi_{z}\D \bar{z})$.

Let $\gamma$ be a real function and $q\in\Gamma(V^{\perp})$ such that
\begin{equation*}
 D_{z}q=2\gamma_{\bar{z}}\kappa-2\gamma D_{\bar{z}}\kappa.
\end{equation*}
Define $\hat{\gamma}$ by
\begin{equation*}
\hat{\gamma}=2\gamma_{zz}+c\gamma+2(q,\kappa)
\end{equation*}
and suppose $\hat{\gamma}$ is real.  Then
$2\hat{\gamma}_{zz}+c\hat{\gamma}$ is also real.
\end{lemma}
\begin{proof}
We compute:
\begin{equation*}
2\hat{\gamma}_{\bar{z}\bar{z}}+\bar{c}\hat{\gamma}=
4\gamma_{zz\bar{z}\bar{z}}+2c_{\bar{z}\bar{z}}\gamma+4c_{\bar{z}}\gamma_{\bar{z}}+2c\gamma_{\bar{z}\bar{z}}+
4(q,\kappa)_{\bar{z}\bar{z}}+
\bar{c}\bigl(2\gamma_{zz}+c\gamma+2(q,\kappa)\bigr).
\end{equation*}
The conformal Gauss equation says that $c_{\bar{z}}=4(\kappa,\kappa)_{z}$ that
$c_{\bar{z}\bar{z}}=4(\kappa,\kappa)_{z\bar{z}}$ is real and we
readily conclude that
\begin{align*}
\im(2\hat{\gamma}_{\bar{z}\bar{z}}+\bar{c}\hat{\gamma})&=
\im\bigl(4c_{\bar{z}}\gamma_{\bar{z}}+4(q,\kappa)_{\bar{z}\bar{z}}+2\bar{c}(q,\kappa)\bigr)\\
&=\im\bigl(16\gamma_{\bar{z}}(\kappa,\kappa)_{z}+4(D^{2}_{\bar{z}\bar{z}}q,\kappa)+8(D_{\bar{z}}q,D_{\bar{z}}\kappa)+
2(q,2D^{2}_{\bar{z}\bar{z}}\kappa+\bar{c}\kappa)\bigr)\\
&=\im\bigl(16\gamma_{\bar{z}}(\kappa,\kappa)_{z}
+4(D^{2}_{\bar{z}\bar{z}}q,\kappa)+8(D_{\bar{z}}q,D_{\bar{z}}\kappa)\bigr),
\end{align*}
thanks to the conformal Codazzi equation.  Now differentiate the
complex conjugate of the equation for $q$ and substitute in to get,
after a short computation:
\begin{equation*}
\im(2\hat{\gamma}_{\bar{z}\bar{z}}+\bar{c}\hat{\gamma})=
\im\bigl(8\gamma_{z\bar{z}}(\kappa,\kappa)+
12\bigl(\gamma_z(\kappa,\kappa)_{\bar{z}}+\gamma_{\bar{z}}(\kappa,\kappa)_z\bigr)\bigr)=0.
\end{equation*}
\end{proof}

\begin{theorem}\label{theorem-existence of a fcq}Let $(\Lambda,\eta)$
be an isothermic surface in $S^{n}$.  Then locally, away from the
zeros of $\eta$, $(\Lambda,\eta)$ has always a formal conserved
quantity.
\end{theorem}
\begin{proof}
We work on a simply connected open set with holomorphic coordinate
$z$ for which $\eta=-\psi\wedge(\psi_{\bar{z}}\D z+\psi_{z}\D
\bar{z})$.  We inductively construct a formal power series
$p(t)=\sum_{i\leq 0} p_{i}t^{i}$ with
\begin{equation*}
p_{i}=\alpha_{i} \psi+\beta_{i}\psi_{z}+\bar{\beta_{i}}\psi_{\bar{z}}+\gamma_{i}\hat{\psi}+q_{i},
\end{equation*}
with $\alpha_i,\gamma_i$ real, $\beta_i$ complex and $q_i$ a section
of $V^{\perp}$ satisfying \eqref{eq:1}.   Then, by
Proposition~\ref{necessary}, $p(t)$ will be a formal conserved quantity.

We begin by taking $\gamma_0=0$ and $q_0$ a non-zero parallel section
of $V^{\perp}$ so that $p_0=q_0$.  The issue is to define $\gamma_i$ and
$q_i$ for then $\alpha_i,\beta_i$ are given
by \eqref{eq:2} and \eqref{eq:3}.  Suppose now that we have
$\gamma_j,q_j$, $j>i\in\Z^{-}$ with
\begin{align*}
D_{z}q_{j}&=2\gamma_{j,\bar{z}}\kappa-2\gamma_{j}D_{\bar{z}}\kappa\\
\gamma_{j}&=2\gamma_{j+1,zz}+c\gamma_{j+1}+2(q_{j+1},\kappa),
\end{align*}
and each $\gamma_j,q_j$ real.  Define $\gamma_{i}$ to be
$2\gamma_{i+1,zz}+c\gamma_{i+1}+2(q_{i+1},\kappa)$ and note that
Lemma~\ref{lemma1} (with $\hat\gamma=\gamma_{i+1}$) tells us that
$\gamma_i$ is real.  Since $D$ is flat, equation \eqref{eq:4} for
$q_i$ is integrable when
$\im\bigl(D_{\bar{z}}(\gamma_{i,\bar{z}}\kappa-\gamma_iD_{\bar{z}}\kappa)\bigr)=0$
however,
\begin{equation*}
\im\bigl(D_{\bar{z}}(\gamma_{i,\bar{z}}\kappa-\gamma_iD_{\bar{z}}\kappa)\bigr)=
\im\bigl(\gamma_{i,\bar{z}\bar{z}}\kappa-\gamma_iD^{2}_{\bar{z}\bar{z}}\kappa\bigr)=
\im\bigl((\gamma_{i,\bar{z}\bar{z}}+\tfrac{\bar{c}}{2}\gamma_i)\kappa\bigr),
\end{equation*}
by the conformal Codazzi equation, and this vanishes thanks to a
second application of Lemma~\ref{lemma1} with $\hat\gamma=\gamma_i$.
Thus, by induction, $\gamma_i,q_i$ are defined for all $i\in\Z^{-}$
satisfying \eqref{eq:4} and \eqref{eq:5} and we are done.
\end{proof}

Theorem \ref{theorem-existence of a fcq} is not completely
satisfying: the result is only local and the quadrature that
determines each $q_i$ means that we lack an explicit formula for the
$p_i$.  More, these quadratures introduce an infinite number of
constants of integration (parallel sections of $V^{\perp}$).

However, the following simple observation allows us to control the
constants of integration: if
$p(t)=\sum_{i\leq 0}p_it^{i}$ is a local formal conserved
quantity for $(\Lambda,\eta)$ with $r(t)=\sum_{i\leq
0}r_tt^i=\bigl(p(t),p(t)\bigr)$ then $r(t)$ is constant by
Proposition~\ref{top term}.  Thus, for all $i$,
\begin{equation}
\label{eq:7}
2(p_0,q_i)=2(p_0,p_i)=r_{i}-\sum_{\substack{k,l<0\\k+l=i}}(p_{k},p_{l}).
\end{equation}
Thus the component of each $q_i$ along $p_0$ is completely determined
up to a constant by the $p_j$ for $0\geq j>i$.

We use this to refine Theorem~\ref{theorem-existence of a fcq}:
\begin{proposition}\label{proposition-existence of formal}
Let $(\Lambda,\eta)$ be isothermic and
$r(t):=\sum_{i\leq 0}
r_{i}t^{i}$ a formal Laurent series with coefficients in
$\mathbb{R}$, such that $r_{0}>0$.  Then, locally, away from zeros of
$\eta$, there is a formal
conserved quantity $p(t)$ of $(\Lambda,\eta)$ such that
$\bigl(p(t),p(t)\bigr)=r(t)$.
\end{proposition}
\begin{proof}
We revisit the induction of Theorem \ref{theorem-existence of a fcq}.
Begin with $\gamma_0=0$ and take $p_0=q_0$ to be a parallel section
of $V^{\perp}$ with $(q_0,q_0)=r_0$.  For the induction step, suppose
we have defined $\gamma_j,q_j$ and so $p_j$, for $0\geq j\geq i$,
with $\sum_{k+l=j}(p_k,p_l)=r_j$ for $j>i$.  We then have that $\D
p_j+\eta p_{j-1}=0$, for $j>i$, while $\D p_i \perp V^{\perp}$ since
$q_i$ solves \eqref{eq:4}.  It follows that $\sum_{k+l=i}(p_k,p_l)$
is constant so that, replacing $p_i$ by $p_i+s_ip_0$, for a suitable
constant $s_i$, we may ensure that $\sum_{k+l=i}(p_k,p_l)=r_i$ also.
\end{proof}

In codimension 1, we can say more: in this case, $p_0$ frames
$V^{\perp}$ so that each $q_i$ is completely determined via
\eqref{eq:7} by $p_{j}$, $j>i$ and $r_{i}$.  It follows at once that
$p(t)$ is uniquely determined in this case on the domain of the
holomorphic coordinate $z$ by $p_0$ and $r(t)$ (and so determined up
to sign by $r(t)$ alone).  We use this to patch together the local
solutions provided by Proposition~\ref{proposition-existence of
formal} to give a global formal conserved quantity away from the zeros
of $\eta$.
\begin{theorem}
\label{theorem-existence of formal}
Let $(\Lambda,\eta)$, $\Lambda:\Sigma\to S^3$ be an isothermic
surface in the $3$-sphere and let $Z\subset \Sigma$ be the (discrete)
zero-set of $\eta$.  Let $r(t)=\sum_{i\leq 0}r_it^i\in\R[[t,t^{-1}]]$
be a formal Laurent series with $r_{0}>0$.

Then there is a formal conserved quantity $p(t)$, unique up to sign,
defined on $\Sigma\setminus Z$ with $\bigl(p(t),p(t)\bigr)=r(t)$.
\end{theorem}
\begin{proof}
Since $\Sigma$ is orientable, $V^{\perp}$ is orientable and so has a
global section $p_0$ with $(p_0,p_0)=r_0$.  We now use
Proposition~\ref{proposition-existence of formal} to cover
$\Sigma\setminus Z$ with open sets $U_{\alpha}$ on which are defined
formal conserved quantities $p^{\alpha}$ with $p_0^{\alpha}=p_0$ and
$\bigl(p^{\alpha}(t),p^{\alpha}(t)\bigr)=r(t)$.  If $x\in
U_{\alpha}\cap U_{\beta}$, then we may find a holomorphic coordinate
$z$ near $x$ for which $\eta=-\psi\wedge(\psi_{\bar{z}}\D z+\psi_z\D
\bar{z})$. Thus, near $x$, $p^{\alpha}_i$ and $p^{\beta}_i$ are
uniquely determined by $p_0$ and $r(t)$ and so must coincide.
\end{proof}

Since holomorphic quadratic differentials on a $2$-torus are
constant, we immediately conclude:
\begin{corollary}\label{th:1}
Let $(\Lambda,\eta)$ be an isothermic $2$-torus in the $3$-sphere and
let $r(t)=\sum_{i\leq 0}r_it^i\in\R[[t,t^{-1}]]$ be a formal Laurent
series with $r_{0}>0$.

Then there is a globally defined formal conserved quantity $p(t)$,
unique up to sign, with $\bigl(p(t),p(t)\bigr)=r(t)$.
\end{corollary}

\section{Special isothermic surfaces of type $d$ in $S^{n}$}
\label{sec:spec-isoth-surf}

In this section, we fix an isothermic surface $(\Lambda,\eta)$ and a
holomorphic coordinate $z$ with $\eta=-\psi\wedge(\psi_{\bar{z}}\D
z+\psi_{z}\D \bar{z})$.

As an application of the ideas of Section~\ref{sec:fcq}, we ask when
$(\Lambda,\eta)$ is special of type $d$.  Our starting point is the
simple observation that $q(t)$ is a polynomial conserved quantity of
degree $d$ then $p(t)=t^{-d}q(t)$ is a formal conserved quantity with
$p_{i}=0$, for all $i<-d$, and conversely.  Moreover, we can choose
the constants of integration so that $p_{i}=0$, for all $i<-d$,
precisely when $\gamma_{-d-1}=0$.  Thus, in view of \eqref{eq:5}, we
have:
\begin{theorem}\label{corollary-arbitrary codimension}
Let $d\in\mathbb{N}_{0}$. $(\Lambda,\eta)$ is a special isothermic
surface of type $d$ if and only if there exists a formal conserved
quantity $p(t)$ of $(\Lambda,\eta)$ such that
\begin{equation*}
2\gamma_{-d,zz}+c\gamma_{-d}+2(q_{-d},\kappa)=0.
\end{equation*}
In this situation, if $\gamma_{-d}\neq 0$, then $(\Lambda,\eta)$ is a
special isothermic surface of type $d$ in $E(p_{-d})$.
\end{theorem}

In this case, with $r(t)=\bigl(p(t),p(t)\bigr)$, we have
$r_{-2d}=(p_{-d},p_{-d})$ and $r_{-2d+1}=2(p_{-d},p_{-d+1})$ and so,
by Proposition~\ref{top term}, these latter inner products are
constant.  Perhaps surprisingly, a converse is available.  First a
lemma:
\begin{lemma}
\label{th:2}
Let $p(t)$ be a formal conserved quantity for $(\Lambda,\eta)$.
Then, for $i,j\leq 0$,
\begin{equation}
\label{eq:14}
(\eta_{\del/\del z}p_{i},p_j)=\gamma_{i,\bar{z}}\gamma_{j}-\gamma_i\gamma_{j,\bar{z}}
\end{equation}
\end{lemma}
\begin{proof}
Recall that $\eta_{\del/\del z}=-\psi\wedge\psi_{\bar{z}}$ while
\begin{equation*}
p_i=\alpha_i\psi-2\gamma_{i,\bar{z}}\psi_z-2\gamma_{i,z}\psi_{\bar{z}}+
\gamma_i\hat{\psi}+q_i
\end{equation*}
so that $\eta_{\del/\del
z}p_i=\gamma_i\psi_{\bar{z}}-\gamma_{i,\bar{z}}\psi$.  Now write $p_j$
in terms of $\gamma_j$ to draw the conclusion.
\end{proof}

With this in hand, we have:
\begin{proposition}
\label{th:3}
$(\Lambda,\eta)$ is locally special isothermic of type at most $d\in\N$ if
and only if it admits a formal conserved quantity for which either
$(p_{-d},p_{-d})$ or, in case $d>1$, $(p_{-d},p_{-d+1})$ is constant.
\end{proposition}
\begin{proof}
We have already seen necessity of the condition on inner products, so
we turn to the sufficiency.   We suppose, without loss of generality,
that $\gamma_{-d}$ is never zero (otherwise
$(\Lambda,\eta)$ is special of type $k< d$) and observe that
$(p_{-d},p_{-d})$ is constant if and only if $(p_{-d},p_{-d})_z=0$.
However,
\begin{equation*}
\begin{split}
\half(p_{-d},p_{-d})_z&=-(\eta_{\del/\del z}p_{-d-1},p_{-d})\\
&=
-\gamma_{-d-1,\bar{z}}\gamma_{-d}+\gamma_{-d-1}\gamma_{-d,\bar{z}}=
-\gamma_{-d}^2\bigl(\frac{\gamma_{-d-1}}{\gamma_{-d}}\bigr)_{\bar{z}},
\end{split}
\end{equation*}
by Lemma~\ref{th:2}.  Thus, when $(p_{-d},p_{-d})$ is constant,
$\gamma_{-d-1}/\gamma_{-d}$ is constant also and we have
$\gamma_{-d-1}=s\gamma_{-d}$, for some $s\in\R$.  Now define a new
formal conserved quantity $\hat{p}(t)=p(t)-st^{-1}p(t)$ and observe
that $\hat{\gamma}_{-d-1}=\gamma_{-d-1}-s\gamma_{-d}=0$ so that
$(\Lambda,\eta)$ is special of type at most $d$.

When $d>1$ and $(p_{-d},p_{-d+1})$ is constant, we argue similarly,
assuming that $\gamma_{-d+1}$ is non-zero and using
\begin{equation*}
(p_{-d},p_{-d+1})_{z}=
-(\eta_{\del/\del z}p_{-d-1},p_{-d+1})-(p_{-d},\eta_{\del/\del z}p_{-d})
=-(\eta_{\del/\del z}p_{-d-1},p_{-d+1}),
\end{equation*}
to conclude that there is a constant $s$ such that
$\gamma_{-d-1}=s\gamma_{-d+1}$ and then work with $p(t)-st^{-2}p(t)$.
\end{proof}

Let us now restrict attention to codimension $1$ where we can
carry out the recursions of Section~\ref{sec:fcq} explicitly.  So
assume that $(\Lambda,\eta)$ is isothermic in $S^3$ and let $N$ be a
unit (hence $D$-parallel) section of the line bundle $V^{\perp}$.
Let $p(t)$ be a formal conserved quantity with
$\bigl(p(t),p(t)\bigr)=r(t)$ and write $\kappa=kN$,
\begin{equation*}
p_i=\alpha_i\psi+\beta_i\psi_z+\overline{\beta_i}\psi_{\bar{z}}+\gamma_i\psi+\delta_iN,
\end{equation*}
where we have introduced real functions $k$ and $\delta_{i}$.

We take, without loss of generality, $p_0=N$ so that $\gamma_0=0$ and
$\delta_0=1$ whence
\begin{subequations}\label{eq:12}
\begin{align}
\label{eq:8}\gamma_{-1}&=2k\\
\label{eq:9}\delta_{-1}&=r_{-1}/2\\
\label{eq:10}p_{-1}&=4(k_{z\bar{z}}+k^3)\psi-4k_{\bar{z}}\psi_z-4k_z\psi_{\bar{z}}+2k\hat{\psi}
+\tfrac{r_{-1}}{2}N\\
\label{eq:11}\gamma_{-2}&=4k_{zz}+(2c+r_{-1})k.
\end{align}
\end{subequations}
As an immediate consequence of Theorem~\ref{corollary-arbitrary
codimension}, we have:
\begin{proposition}{\cite{BurPedPin02}}\label{th:5}
An isothermic surface has constant mean curvature in a
$3$-dimensional space-form if and only there is a constant $H\in\R$
such that
\begin{equation}
\label{eq:13}
2k_{zz}+ck=Hk.
\end{equation}
\end{proposition}
Here, the space-form is $E(p_{-1})$ and $H=-r_{-1}/2$.  In
particular, in this setting $p_{-1}$ and so $(p_{-1},p_{-1})$ is
constant, that is,
\begin{equation*}
16(-k_{z\bar{z}}k-k^4+k_zk_{\bar{z}})+r_{-1}^2/4
\end{equation*}
is constant.  In view of Proposition \ref{th:3}, we deduce the
following result of Musso--Nicolodi:
\begin{proposition}{\cite{MusNic01}}
An isothermic surface has constant mean curvature in a
$3$-dimensional space-form if and only
\begin{equation}
\label{eq:17}
k_{z\bar{z}}k+k^4-k_zk_{\bar{z}}
\end{equation}
is constant.
\end{proposition}

Similarly, Theorem~\ref{corollary-arbitrary codimension} for $d=2$
gives:
\begin{proposition}
\label{th:4}
$(\Lambda,\eta)$ is special isothermic of type $2$ if and only if
there are constants $s_1,s_{2}\in\R$ such that
\begin{equation}
\label{eq:15}
\begin{split}
4k_{zzzz}+4ck_{zz}+4c_zk_z+(2c_{zz}+c^2)k+8(k_{z\bar{z}}k+k^4-k_zk_{\bar{z}})k\\+
s_1(2k_{zz}+ck)+s_2k=0.
\end{split}
\end{equation}
\end{proposition}

\section{Surfaces of revolution, cones and cylinders}\label{surfaces of revolution chapter}

We illustrate the preceding theory by applying it to surfaces of
revolution, cones and cylinders in $S^{3}$.  These surfaces are
automatically isothermic surfaces and we will find necessary and
sufficient conditions for them to be special isothermic.

\subsection{Surfaces of revolution and cones}

We take a uniform approach to cones and surfaces of revolution by
viewing them as extrinsic products.  For this, let
$W\leq\mathbb{R}^{4,1}$ be a $3$-dimensional subspace and contemplate
the direct sum
\begin{equation*}
\mathbb{R}^{4,1}=W\oplus W^{\perp}.
\end{equation*}
Let $\Sigma=I_1\times I_2$ be a product of intervals and suppose that
$\Lambda:\Sigma\to S^3=\mathbb{P}(\mathcal{L})$ is of the form
\begin{equation*}
\Lambda:=\langle \phi_{1}+\phi_{2}\rangle
\end{equation*}
where $\phi_1:I_1\to W$ and $\phi_2:I_2\to W^{\perp}$ are curves with
$(\phi_1,\phi_1)=-(\phi_2,\phi_2)=C$, for a non-zero constant
$C\in\R$.

Here is the geometry of the situation: if $W$ has indefinite
signature (so that $C<0$), then $\phi_1$ takes values in a
hyperboloid while $\phi_2$ is circle-valued.  Taking the half-plane
model of the hyperboloid, we see that $\Lambda$ is a surface of
revolution.  Similarly, if $C>0$, $\phi_1$ is a curve on a $2$-sphere
while $\phi_2$ takes values in a half-line so that $\Lambda$ is the
cone over $\phi_1$.

All such surfaces are isothermic: one easily checks that
$\eta=(\D\phi_{1}-\D\phi_{2})\wedge(\phi_{1}+\phi_{2})$ is closed.
The corresponding holomorphic coordinate is also easy to identify:
with $u,v$ the arc-length parameters on $I_1,I_2$ respectively, set
$z=u+iv$.  Then $\psi:=\phi_{1}+\phi_{2}$ is satisfies $|\D
\psi|^{2}=|\D z|^{2}$ and $\eta=-\psi\wedge(\psi_{\bar{z}}\D
z+\psi_{z}\D \bar{z})$.

Set $S=\{x\in W:(x,x)=C\}$ so that $\phi_1:I_1\to S$, let $n$ be a
unit normal to $\phi_1$ in $S$ and $\bk$ the corresponding curvature
so that
\begin{equation*}
\phi_{1}''=\bk n-\frac{1}{C}\phi_{1}.
\end{equation*}
We also have:
\begin{equation*}
\phi_{2}''=\frac{1}{C}\phi_{2}.
\end{equation*}
Using these, we compute:
\begin{subequations}
\label{eq:16}
\begin{align}
\psi_z&=\half(\phi_1'-i\phi_2')\\
N&=n+\frac{\bk}{2}(\phi_1+\phi_2)\\
\hat{\psi}&=\half(\tfrac{\bk^2}{4}-\tfrac{1}{C})\phi_1+
\half(\tfrac{\bk^2}{4}+\tfrac{1}{C})\phi_2+\tfrac{\bk}{2}n
\end{align}
\end{subequations}
with $N$ a unit section of $V^{\perp}$, for $V$ the central sphere
congruence of $\Lambda$.  Then
\begin{equation*}
\psi_{zz}=\frac{1}{4}(\phi_{1}''-\phi_{2}'')=
-\frac{1}{4C}(\phi_{1}+\phi_{2})+\frac{\bk}{4}n=
(-\frac{1}{4C}-\frac{\bk^{2}}{8})\psi+\frac{\bk}{4}N,
\end{equation*}
so that the Schwarzian derivative and Hopf differential of $\Lambda$
are given by
\begin{equation*}
c=\frac{1}{2C}+\frac{\bk^{2}}{4}\qquad\kappa=\frac{\bk}{4}N.
\end{equation*}

In the notation of section~\ref{sec:spec-isoth-surf}, $k=\bk/4$ so
that, in the current setting, Proposition \ref{th:5} reads
\begin{proposition}
$(\Lambda,\eta)$ is special isothermic of type $1$ if and only if,
for some constant $\alpha$,
\begin{equation}\label{degree 1 of s.revolution/cones}
\frac{\bk}{C}+\frac{\bk^{3}}{2}+\bk''+\alpha \bk=0;
\end{equation}
\end{proposition}
Note that equation (\ref{degree 1 of s.revolution/cones}) means
exactly that $\phi_{1}$ is an elastic curve in $S$, and then
$\alpha=0$ if and only if $\phi_{1}$ is a free elastic curve.

Similarly, Proposition~\ref{th:4} reads:
\begin{proposition}
$(\Lambda,\eta)$ is special isothermic of type $2$ if and
only if there exist real constants $\alpha$ and $\beta$ such that
\begin{equation*}
\frac{\bk}{C^{2}}+\frac{\bk^{3}}{C}+\frac{3\bk^{5}}{8}+\frac{2\bk''}{C}
+\frac{5}{2}(\bk\bk'^{2}+\bk^{2}\bk'')+\bk^{(\mathrm{iv})}
+\alpha(\frac{\bk}{C}+\frac{\bk^{3}}{2}+\bk'')+{\beta\bk}=0.
\end{equation*}
\end{proposition}

We now prove that the constant term $p_{-d}$ of a polynomial
conserved quantity of $(\Lambda,\eta)$ lies $W$. For this, we recall
the notations of the previous sections and begin with a lemma:
\begin{lemma}\label{lemma}
Let $p(t)$ be a formal conserved quantity for $(\Lambda,\eta)$.
Then, for each $i\leq 0$, $\gamma_{i}=-(\psi,p_i)$ is independent of $v$.
\begin{proof}
We induct, noting that both $c$ and $k=\bk/4$ are independent of $v$.
\end{proof}
\end{lemma}

\begin{proposition}\label{constant terms}The constant terms of the
polynomial conserved quantities of $(\Lambda,\eta)$ lie in $W$.
\end{proposition}
\begin{proof}
Suppose that $(\Lambda,\eta)$ is a special isothermic surface of type
$d$ with polynomial conserved quantity $q(t)$ and work with the
formal conserved quantity $p(t)=t^{-d}q(t)$.  Observe that
Lemma~\ref{lemma} gives
\begin{equation*}
0=(p(t),\psi)_v=-t(\eta_{\del/\del v}p(t),\psi)+(p(t),\psi_{v})=(p(t),\phi_2'),
\end{equation*}
since $\eta\psi=0$.  In particular, $(p_{-d},\phi_2')=0$.  Moreover,
from \eqref{eq:16}, we have
\begin{equation*}
(p_{-d},\phi_{2})=
-\alpha_{-d}C+\gamma_{-d}(-\frac{\bk^{2}}{8}C-\frac{1}{2})+
\delta_{-d}(-\frac{1}{2}\bk C).
\end{equation*}
However, in this context, by Lemma~\ref{lemma}, \eqref{eq:3} reads
\begin{equation*}
\alpha_{-d}=\half\gamma_{-d,uu}+\frac{\bk^2}{8}\gamma_{-d}=
2\gamma_{-d,zz}+\frac{\bk^2}{8}\gamma_{-d}
\end{equation*}
so that
\begin{equation*}
(p_{-d},\phi_{2})=-C(2\gamma_{-d,zz}+c\gamma_{-d}+2\delta_{-d}\frac{\bk}{4})=0.
\end{equation*}
Since $\phi_2,\phi_2'$ frame $W^{\perp}$, the result follows.
\end{proof}

For cones, this has a geometric consequence: if $p_{-d}\neq 0$, we
have $(p_{-d},p_{-d})>0$, since $W$ has definite signature, so that
$(\Lambda,\eta)$ is special isothermic in a hyperbolic space form.

\subsection{Cylinders}

A similar analysis may be carried out for cylinders which amount to
the limiting case where the curvature $1/C$ of $S$ is
zero.  We briefly rehearse the details.

Let $v_{0}, v_{\infty}\in\mathcal{L}$ with
$(v_{0},v_{\infty})=-1$, choose a $2$-dimensional subspace
$U\leq\langle v_0,v_{\infty}\rangle^{\perp}$ and write
\begin{equation*}
\langle v_{0},v_{\infty}\rangle^{\perp}=U\oplus U^{\perp}
\end{equation*}

Again let $\Sigma=I_1\times I_2$ be a product of intervals and
suppose that $\Lambda:\Sigma\to S^{3}\cong\mathbb{P}(\mathcal{L})$ is
of the form
\begin{equation*}
\Lambda:=\langle \phi_{1}+\phi_{2}+v_{0}+\frac{1}{2}\big((\phi_{1},\phi_{1})+(\phi_{2},\phi_{2})\big)v_{\infty}\rangle
\end{equation*}
where $\phi_{1}:I_{1}\to U$ and $\phi_{2}:I_{2}\to U^{\perp}$ are
curves.  Let $u,v$ be arc-length parameters on $I_1,I_2$ respectively
and set $z=u+iv$.  Then
$\psi=\phi_{1}+\phi_{2}+v_{0}+\frac{1}{2}\big((\phi_{1},\phi_{1})+(\phi_{2},\phi_{2})\big)v_{\infty}$
has $|\D \psi|^{2}=|\D z|^{2}$ and is isothermic with
\begin{equation*}
\eta=\bigl((\D\phi_1+(\D\phi_1,\phi_1)v_{\infty})-
(\D\phi_2+(\D\phi_2,\phi_2)v_{\infty})\bigr)\wedge\psi.
\end{equation*}

Let $n$ be a unit normal to $\phi_1$ in $U$ with corresponding
curvature $\bk$ so that $\phi_1''=\bk n$.  Using $\phi_2''=0$, we
have
\begin{subequations}
\label{eq:23}
\begin{align}
\label{eq:18}
\psi_z&=\half\bigl((\phi_1'+(\phi_1',\phi_1)v_{\infty})-i
(\phi_2'+(\phi_2',\phi_2)v_{\infty})\bigr)\\
\label{eq:19}N&=n+(n,\phi_{1})v_{\infty}+\frac{\bk}{2}\psi\\
\label{eq:20}\hat{\psi}&=
v_{\infty}+\frac{\bk}{2}(n+(n,\phi_1)v_{\infty})+\frac{\bk^2}{8}\psi
\end{align}
\end{subequations}
and then
\begin{equation*}
\psi_{zz}=\frac{\bk}{4}(n+(n,\phi_1)v_{\infty})=-\frac{\bk^2}{8}\psi+\frac{\bk}{4}N
\end{equation*}
so that
\begin{equation*}
c=\frac{\bk^2}{4}\qquad \kappa=\frac{\bk}{4}N.
\end{equation*}
We conclude:
\begin{proposition}
A cylinder $(\Lambda,\eta)$ is special isothermic of type $1$ if and only if,
for some constant $\alpha$,
\begin{equation*}
\frac{\bk^{3}}{2}+\bk''+\alpha \bk=0;
\end{equation*}
\end{proposition}
\begin{proposition}
A cylinder $(\Lambda,\eta)$ is special isothermic of type $2$ if and
only if there exist real constants $\alpha$ and $\beta$ such that
\begin{equation*}
\frac{3\bk^{5}}{8}+
+\frac{5}{2}(\bk\bk'^{2}+\bk^{2}\bk'')+\bk^{(\mathrm{iv})}
+\alpha(\frac{\bk^{3}}{2}+\bk'')+{\beta\bk}=0.
\end{equation*}
\end{proposition}

The analogue of Lemma~\ref{lemma} holds with the same argument:
\begin{lemma}\label{th:6}
Let $p(t)$ be a formal conserved quantity for $(\Lambda,\eta)$.
Then, for each $i\leq 0$, $\gamma_{i}=-(\psi,p_i)$ is independent of $v$.
\end{lemma}

\begin{proposition}\label{constant terms-cylinders}The constant terms
of the polynomial conserved quantities of $(\Lambda,\eta)$ lie in
$U\oplus\langle v_{\infty}\rangle$.
\end{proposition}
\begin{proof}
Suppose that $(\Lambda,\eta)$ is a special isothermic surface of type
$d$ with polynomial conserved quantity $q(t)$ and work with the
formal conserved quantity $p(t)=t^{-d}q(t)$.  From
Lemma~\ref{th:6}, we have
\begin{equation}\label{eq:21}
0=(p(t),\psi)_v=-t(\eta_{\del/\del v}p(t),\psi)+(p(t),\psi_{v})=
(p(t),\phi_2')+(\phi_2',\phi_{2})(p(t),v_{\infty}).
\end{equation}
Now $(\psi,v_{\infty})=-1$ so that $(\D\psi,v_{\infty})=0$ while
\eqref{eq:20} yields $(\hat{\psi},v_{\infty})=-\bk^2/8$.  Thus,
\begin{equation*}
(p_{-d},v_{\infty})=
-\alpha_{-d}-\gamma_{-d}\frac{\bk^{2}}{8}-\delta_{-d}\frac{\bk}{2}.
\end{equation*}
Once more we have
\begin{equation}\label{eq:22}
\alpha_{-d}=\half\gamma_{-d,uu}+\frac{\bk^2}{8}\gamma_{-d}=
2\gamma_{-d,zz}+\frac{\bk^2}{8}\gamma_{-d}
\end{equation}
so that
\begin{equation*}
(p_{-d},v_{\infty})=
-2\gamma_{-d,zz}-\gamma_{-d}\frac{\bk^{2}}{4}-\delta_{-d}\frac{\bk}{2}=0.
\end{equation*}
Now \eqref{eq:21} yields $(p_{-d},\phi_2')=0$ and, since
$v_{\infty},\phi_2'$ frame $(U\oplus\langle
v_{\infty}\rangle)^{\perp}$, we are done.
\end{proof}

Again this has a geometric consequence: if $p_{-d}\neq 0$,
$(p_{-d},p_{-d})\geq 0$ with equality if and only if $p_{-d}\in\langle
v_{\infty}\rangle$.  Thus a special isothermic cylinder is either
special isothermic in a hyperbolic space or in the particular
Euclidean space $E(v_{\infty})$.

In this last case, the conditions to be special isothermic of type
$d$ simplify considerably:
\begin{proposition}
Let $d\in\mathbb{N}_{0}$. A cylinder $(\Lambda,\eta)$ is a special
isothermic surface of type $d$ in $E(v_{\infty})$ if and only if there
exists a formal conserved quantity $p(t)$ of $(\Lambda,\eta)$ such
that $\gamma_{-d}$ is constant.
\end{proposition}
\begin{proof}
Certainly, if we have $p(t)$ with $p_{-d}$ lying in $\langle
v_{\infty}\rangle$ then $0=(p_{-d},\psi_z)=-2\gamma_{-d,z}$ so
that $\gamma_{-d}$ is constant.

For the converse, suppose that $\gamma_{-d}$ is constant and note
that, in the present context, \eqref{eq:4} yields
\begin{equation*}
\delta_{-d,z}=-\gamma_{-d}\bk'/4
\end{equation*}
so that we may take $\delta_{-d}=-\gamma_{-d}\bk/2$.  On the other
hand, by \eqref{eq:20} and \eqref{eq:22}, we have
\begin{equation*}
p_{-d}=2\gamma_{-d,zz}\psi-2\gamma_{-d,\bar{z}}\psi_z-2\gamma_{-d,z}\psi_{\bar{z}}+(\delta_{-d}+\frac{\bk}{2}\gamma_{-d})N+\gamma_{-d}v_{\infty}
\end{equation*}
and all coefficients except the last vanish so we are done.
\end{proof}
In particular, $(\Lambda,\eta)$ is a special isothermic
surface of type $1$ in $E(v_{\infty})$ if
and only if
\begin{equation*}
\bk \mbox{ is constant};
\end{equation*}
$(\Lambda,\eta)$ is a special isothermic surface of type $2$ in
$E(v_{\infty})$ if and only if there exists a real constant $\alpha$
such that
\begin{equation}
\frac{\bk^{3}}{2}+\bk''+\alpha \bk\mbox{ is constant};
\end{equation}
$(\Lambda,\eta)$ is a special isothermic surface of type $3$ in
$E(v_{\infty})$ if and only if there exist real constants $\alpha$
and $\beta$ such that
\begin{equation*}
\frac{3\bk^{5}}{8}+
+\frac{5}{2}(\bk\bk'^{2}+\bk^{2}\bk'')+\bk^{(\mathrm{iv})}
+\alpha(\frac{\bk^{3}}{2}+\bk'')+{\beta\bk} \mbox{ is constant}.
\end{equation*}


\end{document}